\documentclass[10pt]{amsart}
\usepackage[margin=1.2in,marginparsep=0.1in,marginparwidth=1in]{geometry}
\usepackage{amssymb,amsmath,amsthm,amstext,amscd,latexsym,graphics,graphicx,bbm,caption}
\usepackage[dvipsnames,svgnames,table]{xcolor}
\usepackage[plainpages=false,colorlinks=true, pagebackref]{hyperref}
\usepackage{tikz}
\usepackage{mathtools}
\usepackage{tikz-cd}

\hypersetup{citecolor=Sepia,linkcolor=blue, urlcolor=blue}

\usepackage{cite}  

\makeatletter
\def\@tocline#1#2#3#4#5#6#7{\relax
	\ifnum #1>\c@tocdepth 
	\else
	\par \addpenalty\@secpenalty\addvspace{#2}%
	\begingroup \hyphenpenalty\@M
	\@ifempty{#4}{%
		\@tempdima\csname r@tocindent\number#1\endcsname\relax
	}{%
		\@tempdima#4\relax
	}%
	\parindent\z@ \leftskip#3\relax \advance\leftskip\@tempdima\relax
	\rightskip\@pnumwidth plus4em \parfillskip-\@pnumwidth
	#5\leavevmode\hskip-\@tempdima
	\ifcase #1
	\or\or \hskip 2em \or \hskip 2em \else \hskip 3em \fi%
	#6\nobreak\relax
	\dotfill\hbox to\@pnumwidth{\@tocpagenum{#7}}\par
	\nobreak
	\endgroup
	\fi}
\makeatother
\newtheorem{intro-thm}{Theorem}[]
\theoremstyle{plain}
\newtheorem{thm}{Theorem}[section]
\newtheorem{theorem}[thm]{Theorem}
\newtheorem{question}[thm]{Question}
\newtheorem{lemma}[thm]{Lemma}
\newtheorem{corollary}[thm]{Corollary}
\newtheorem{proposition}[thm]{Proposition}

\theoremstyle{definition}
\newtheorem{remark}[thm]{Remark}

\newtheorem{definition}[thm]{Definition}

\newcommand{\Gal}{{\rm Gal}}

\renewcommand{\P}{{\mathbb P}}

\newcommand{\Z}{{\mathbb Z}}

\newcommand{\Br}{\text{Br}}

\input{xy}

\begin{document}
	\title{Zero cycles on Severi--Brauer flag varieties}

    \author[D. C-Ramachandran]{Divyasree C-Ramachandran} \address{Department of Mathematics, IISER Pune, Dr Homi Bhabha Rd, Pashan, Pune, 411008, India}\email{crdivya99@gmail.com}
    \author[A. Hogadi]{Amit Hogadi} \address{Department of Mathematics, IISER Pune, Dr Homi Bhabha Rd, Pashan, Pune, 411008, India}\email{amit@iiserpune.ac.in}
	
	\thanks{D. C-Ramachandran was partially supported by the DST-INSPIRE Fellowship (Reg.No. IF210208).}

	\begin{abstract}
		Let \(A\) be a central simple algebra over a field \(F\) with index \(n\) and let \(\mathrm{SB}_r(A)\) denote the \(r\)-th generalized Severi--Brauer variety associated with \(A\). We prove that the Chow group of zero cycles of degree zero \(\mathrm{A_0}(\mathrm{SB}_r(A))\) is \((d, n/d)\)-torsion where \(d = (r,n)\). Our approach reduces the general case to division algebras of prime power index and yields several new instances in which \(\mathrm{A_0}\) is trivial, together with sharper torsion bounds in general.\\ We also show that if \(F\) is a local or global field, then \(\mathrm{A_0}(\mathrm{SB}_r(A))=0\). Since Severi--Brauer flag varieties are stably birational to generalized Severi--Brauer varieties, these results extend to them, yielding corresponding torsion bounds and vanishing results for \(\mathrm{A_0}(X)\), where \(X\) is stably birational to \(\mathrm{SB}_r(A)\).
	\end{abstract}
	
	\maketitle
\section{Introduction}
Let \(X\) be a proper variety over a field \(F\), and let \(\mathrm{CH_0}(X)\) denote the Chow group of zero-cycles on \(X\). Let \(\mathrm{A_0}(X)\) denote the kernel of the degree map \(\deg : \mathrm{CH_0}(X) \rightarrow \Z \). Panin proved in his thesis that \(\mathrm{A_0}(\mathrm{SB}(A))=0\) for the Severi--Brauer variety \(\mathrm{SB}(A)\) associated with a central simple algebra \(A\), provided the characteristic of the base field does not divide the index of \(A\). Severi--Brauer varieties are twisted forms of projective space and constitute an important class of projective homogeneous varieties. A natural generalization is given by the class of generalized Severi--Brauer varieties, which are twisted forms of Grassmannians (see \cite{blanchet} for further reading). More generally, one can consider Severi--Brauer flag varieties (see Definition~\ref{flag}), which are twisted forms of flag varieties. Given a central simple algebra \(A\) of degree \(m\) over a field \(F\), and integers \(0 < n_1 < \cdots < n_k < m\), one associates a Severi--Brauer flag variety of type \((n_1,\ldots,n_k)\), denoted \(\mathrm{SB}_{n_1,\ldots,n_k}(A)\).\\
Let \(d = g.c.d(n_1,n_2,\ldots,n_k,\operatorname{ind}(A))\). In \cite{zerocycle}, Krashen showed that the group \(\mathrm{A_0}\bigl(\mathrm{SB}_{n_1,\ldots,n_k}(A)\bigr)\) is trivial when \(d=1\), under the assumption that the base field \(F\) is perfect or that \(char (F)\nmid \operatorname{ind}(A)\). In particular, this recovers Panin’s earlier result for Severi--Brauer varieties. Krashen's proof, based on étale algebra techniques and symmetric power constructions, also yields the triviality of \(\mathrm{A_0}(\mathrm{SB}_{n_1,\ldots,n_k}(A))\) in the case \(d=2\), when the exponent of \(A\) divides \(2\) or the index of \(A\) divides 4. Subsequently, in \cite{MC}, Merkurjev and Chernousov removed the restrictions on the base field and proved the triviality of \(\mathrm{A_0}\) for a large class of projective homogeneous varieties, including Severi–Brauer varieties and Severi--Brauer flag varieties in the case \(d=2\) where the exponent of \(A\) is not divisible by \(4\), using the theory of connectedness of classes of fields. This motivates the following question.
\begin{question}\label{conjecture}
    Is the Chow group of zero cycles of degree zero of a Severi--Brauer flag variety trivial?
\end{question}
We have an affirmative answer for local and global fields.
\begin{theorem}\label{vanishing over local and global}
Let \(F\) be a local or global field, and let \(X = \mathrm{SB}_{n_1,n_2,\dots,n_k}(A)\) be a Severi–Brauer flag variety over \(F\). Then \(\mathrm{A_0}(X)=0\).
\end{theorem}
This result is proved using the theory of equivalence of fields developed in \cite{MC}. The key property of local and global fields used in the argument is formulated in Lemma~\ref{local-global}. To address Question~\ref{conjecture} in general, we study torsion bounds for \(\mathrm{A_0}\) of Severi--Brauer flag varieties and prove the following theorem.
\begin{theorem}\label{main theorem}
 Let \(A\) be a central simple algebra of index \(n\) over a field \(F\). Let \(\{n_1,n_2,\dots,n_k\}\) be integers satisfying \(0<n_1 < n_2 < \dots < n_k < \text{deg}(A)\). Consider the associated Severi--Brauer flag variety \(X = \mathrm{SB}_{n_1,n_2,\dots,n_k}(A)\). Let \(d=gcd(n_1,n_2,\dots,n_k,n)\). Then \(\mathrm{A_0}(X)\) is \((d,n/d)\)-torsion.
\end{theorem}
This theorem yields several new instances in which Question~\ref{conjecture} admits an affirmative answer. We record some of these consequences in the following corollary and in Proposition~\ref{2-primary part is trivial}.
\begin{corollary}\label{coro}
Let \(A\) be a central simple algebra over \(F\) of index \(n\). If \(n\) is square-free, then \(\mathrm{A_0}(\mathrm{SB}_r(A)) = 0\) for every \(r\).
\end{corollary}
If \(X\) is a geometrically rational variety over a field \(F\),  then \(\mathrm{A_0}(X)\) is a torsion group. In particular, \(\mathrm{A_0}(X)\) is annihilated by the index of \(X\). In Theorem~\ref{index of generalized SB}, we generalize \cite[Lemma 7.1]{zerocycle} by removing the assumption on characteristic of the base field and show that the index of \(\mathrm{SB}_{n_1,\ldots,n_k}(A)\) is \(n/d\), where \(n = \mathrm{ind}(A)\) and \(d = (n_1,n_2,\dots,n_k,n)\). This immediately yields one part of our Theorem~\ref{main theorem}.\\
The remaining assertion, namely that \(\mathrm{A_0}(\mathrm{SB}_{n_1,\ldots,n_k}(A))\) is \(d\)-torsion, is proved in Section~\ref{proof of main thm}.
A key step in our approach is showing that, in order to understand \(\mathrm{A_0}(\mathrm{SB}_r(A))\) for a general central simple algebra \(A\), it suffices to study the groups \(\mathrm{A_0}(\mathrm{SB}_{p^a}(D))\), where \(D\) is a central division algebra of prime power degree.

\subsection{Overview}
We begin by fixing notation and recalling background on Severi--Brauer flag varieties in Section~\ref{prelims}. In Section~\ref{index of SB flag variety}, we show that 
\(ind\big(\mathrm{SB}_{n_1,\ldots,n_k}(A)\big)=n/d\), where \(n=\mathrm{ind}(A)\) and \(d=\gcd(n_1,\ldots,n_k,n)\) (Theorem~\ref{index of generalized SB}), extending \cite[Lemma~7.1]{zerocycle} to arbitrary characteristic. Section~\ref{proof of main thm} contains the proofs of Theorem~\ref{main theorem} and Corollary~\ref{coro}, and also includes Proposition~\ref{2-primary part is trivial}, which gives a vanishing result for certain generalized Severi--Brauer varieties of index divisible by \(2\) but not by \(8\). In Section~\ref{Refined torsion bounds via products}, we develop an inductive approach toward Question~\ref{conjecture} (see Corollary~\ref{reduction in torsion}) and apply it to show that \(\mathrm{A_0}(\mathrm{SB}_r(A))\) is \(2\)-torsion whenever \((r,\mathrm{ind}(A))=4\) and \(4 \nmid \exp(A)\) (Corollary~\ref{2-torsion}). In Section~\ref{Index reduction over local and global fields}, we recall basic properties of Brauer groups over local and global fields and prove a key index-reduction result (Lemma~\ref{local-global}) used in the proof of Theorem~\ref{vanishing over local and global}. Finally, in Section~\ref{proof of Vanishing theorem}, we establish the required equivalence of fields over local and global fields by induction on the degree of the intersection of the extensions (Lemma~\ref{p-special connectedness-global}) and complete the proof of Theorem~\ref{vanishing over local and global}.

\section{Preliminaries on Severi--Brauer flag varieties}\label{prelims}
Let \(F\) be a field. A variety \(X\) over \(F\) is an integral scheme of finite type over \(F\); \(F(X)\) denotes the function field of \(X\).  If \(X\) is an \(F\)-variety and \(L\) is an extension of \(F\), then \(X_L:= X \times_{Spec \, F} Spec \, L\). Two varieties \(X,Y\) over \(F\) are said to be stably birational, if there exist non-negative integers \(m,l\) such that \(X \times \P^m\) is birational to \(Y \times \P^l\). \\
Recall that for any central simple algebra over a field, there exists a unique (upto isomorphism) central division algebra over the same field that is Brauer-equivalent to it, and the index of the algebra is the degree of this division algebra. Moreover, every central division algebra admits a primary decomposition.
\begin{proposition}[{\cite[Proposition 4.5.16]{Gille_Szamuely_2006}}]\label{primary decomposition}
  Let \(D\) be a central division algebra over \(F\) whose index admits the prime power
decomposition \(n = p_1^{a_1} p_2^{a_2} \cdots p_k^{a_k}.\) Then there exist central division algebras \(D_i\) for \(1 \leq i \leq k\) such that \(D \cong D_1 \otimes_F D_2 \otimes_F \dots \otimes_F D_k\) with \(\mathrm{ind}(D_i)=p_i^{a_i}\) for each \(i\).  
\end{proposition}
We refer the reader to \cite{Gille_Szamuely_2006} for further background on the theory of central simple algebras and Severi–Brauer varieties. Next, we recall the definition of a Severi--Brauer flag variety.

\begin{definition}[{\cite[Definition 4.1]{zerocycle}}]\label{flag}
Let \(A\) be a central simple algebra of degree \(n\) over a field \(F\), and let
\(0 < n_1 < \cdots < n_k < n\) be integers.
The Severi--Brauer flag variety of type \((n_1,\ldots,n_k)\), denoted by
\(\mathrm{SB}_{n_1,\ldots,n_k}(A)\), is the variety whose points correspond to flags of
right ideals \(I_{n_1} \subset I_{n_2} \subset \cdots \subset I_{n_k}\) of \(A\), where each \(I_{n_i}\) has reduced dimension \(n_i\).
More precisely, \(\mathrm{SB}_{n_1,\ldots,n_k}(A)\) represents the functor which assigns to
each \(F\)-algebra \(R\) the set
\[
\mathrm{SB}_{n_1,\ldots,n_k}(A)(R)
=
\bigl\{ (I_1,\ldots,I_k) \mid
I_i \in \mathrm{SB}_{n_i}(A)(R),\; I_i \subset I_{i+1} \bigr\},
\]
where \(I_i\) is a right ideal of \(A_R = A \otimes_F R\) that is locally a direct
summand of reduced rank \(n_i\).
\end{definition}
In particular, when \(k=1\), \(\mathrm{SB}_r(A)\) denotes the \(r\)-th generalized Severi--Brauer variety associated with \(A\), parametrizing right ideals of \(A\) of reduced dimension \(r\). When \(r=1\), we have the classical Severi--Brauer variety associated with \(A\). Let \(K/F\) be a splitting field of \(A\), then \(\mathrm{SB}_r(A) \otimes_F K \cong \operatorname{Gr}(r,n)_K\). Hence each generalized Severi--Brauer variety is a form of Grassmanian. For further background, see \cite{blanchet}. The existence of a rational point in a generalized Severi--Brauer variety is characterized by the following lemma.
\begin{lemma}[Blanchet, {\cite[Proposition 3]{blanchet}}]\label{point,index and rational}
Let \(A\) be a central simple algebra over \(F\). Then the following are equivalent:
\begin{enumerate}
    \item The variety \(\mathrm{SB}_r(A)\) has an \(L\)-rational point.
    \item \(\mathrm{ind}(A \otimes_F L) \mid r\).
    \item \(\mathrm{SB}_r(A)_L\) is rational.
\end{enumerate}
\end{lemma}
\begin{remark}\label{point on flag variety}
An analogous statement holds for Severi--Brauer flag varieties. More precisely, let \(\mathrm{SB}_{n_1,\ldots,n_k}(A)\) be a Severi--Brauer flag variety. Then \(\mathrm{SB}_{n_1,\ldots,n_k}(A)(L)\neq \emptyset\) if and only if \(\mathrm{ind}(A_L)\mid d\), where \(d=gcd(n_1,\ldots,n_k,\mathrm{ind}(A))\) (see \cite[Equation 5.3]{MPW}). In this case, \(\mathrm{SB}_{n_1,\ldots,n_k}(A)_L \simeq GL_n/P\) for a parabolic subgroup \(P \subset GL_n\) defined over \(L\); in particular, it is rational by \cite[Theorem~21.20]{borel1991linear}.
\end{remark}
 We conclude this section with a description of the behavior of the index of a central simple algebra under base change to the function field of a Severi–-Brauer flag variety.
\begin{lemma}[Merkurjev, Panin, Wordsworth,{\cite[Equation 5.11]{MPW}}]\label{index after base change}
    Let \(\mathrm{SB}_{n_1,\ldots,n_k}(A)\) be a Severi--Brauer flag variety over \(F\). The index of \(A\) over \(F(\mathrm{SB}_{n_1,\ldots,n_k}(A))\) is given by 
    \[\mathrm{ind}(A \otimes_F F(\mathrm{SB}_{n_1,\ldots,n_k}(A))) = g.c.d(\mathrm{ind}(A),n_1,n_2,\dots,n_k).\]
\end{lemma}

\section{Index of Severi--Brauer flag varieties}\label{index of SB flag variety}
The goal of this section is to show that the index of \(\mathrm{SB}_{n_1,\ldots,n_k}(A)\) is \(n/d\), where \(n=\mathrm{ind}(A)\) and \(d=\gcd(n_1,\ldots,n_k,n)\) (see Theorem~\ref{index of SB flag variety}), thereby extending \cite[Lemma~7.1]{zerocycle} to arbitrary characteristic. Since the index of a variety is a stable birational invariant, we reduce the problem to a generalized Severi--Brauer variety. We begin by recalling a standard criterion for establishing stable birationality between varieties. 

\begin{lemma}\label{stable birationality}
Suppose that \(X\) and \(Y\) are proper varieties over \(F\) such that \(X_{F(Y)}\) and \(Y_{F(X)}\) are rational. Then \(X\) and \(Y\) are stably birational.
\end{lemma}
\begin{proof}
Since \(X_{F(Y)}\) is rational, we have \(F(X \times_F Y) \cong F(Y)(t_1,\dots,t_r)\) with \(r=\dim X\). Similarly, \(F(X \times_F Y) \cong F(X)(s_1,\dots,s_s)\) with \(s=\dim Y\). Thus
\(X \times \mathbb{P}^s \sim_{\mathrm{bir}} X \times_F Y \sim_{\mathrm{bir}} Y \times \mathbb{P}^r\).
\end{proof}

\begin{lemma}\label{Severi--Brauer flag to generalized Severi--Brauer}
     Let \(X\) denote the Severi--Brauer flag variety \(\mathrm{SB}_{n_1,n_2,\dots,n_k}(A)\) associated with a central simple algebra \(A\) of index \(n\) over \(F\). Let \( d=gcd(n_1,n_2,\dots,n_k,n)\). Then \(X\) is stably birational to \(\mathrm{SB}_d(D)\) where \(D\) is the central division algebra over \(F\) Brauer-equivalent to \(A\).
\end{lemma}
\begin{proof}
We first show that \(X\) is stably birational to \(\mathrm{SB}_d(A)\). By Lemma~\ref{index after base change}, we have \(\mathrm{ind}(A\otimes F(\mathrm{SB}_d(A))=d\). Since \(d \mid n_i\) for each \(1 \le i \le k\), the variety \(X\) has a rational point over \(F(\mathrm{SB}_d(A))\), and hence \(X_{F(\mathrm{SB}_d(A))}\) is rational by Remark~\ref{point on flag variety}. Similarly, \(\mathrm{SB}_d(A)_{F(X)}\) is rational, as \(\mathrm{ind}(A \otimes F(\mathrm{SB}_{n_1,n_2,\dots,n_k}(A)))= g.c.d (n,n_1,n_2,\dots,n_k)=d\). Thus \(X\) and \(\mathrm{SB}_d(A)\) are stably birational by Lemma~\ref{stable birationality}. We now compare \(\mathrm{SB}_d(A)\) and \(\mathrm{SB}_d(D)\). Since \(\mathrm{ind}(A \otimes_F E)=\mathrm{ind}(D \otimes_F E)\) for any field extension \(E/F\), both \(\mathrm{SB}_d(D)_{F(\mathrm{SB}_d(A))}\) and \(\mathrm{SB}_d(A)_{F(\mathrm{SB}_d(D))}\) are rational varieties. Hence, \(\mathrm{SB}_d(D)\) and \(\mathrm{SB}_d(A)\) are stably birational by Lemma~\ref{stable birationality}, and the claim follows.
\end{proof}
The following lemma is extracted from Krashen's proof of the index of a generalized Severi--Brauer variety in \cite{zerocycle}. It does not impose any restriction on the characteristic of the base field. We include the argument for the reader's convenience.

\begin{lemma}[Krashen, see the proof of {\cite[Lemma~7.1]{zerocycle}}]\label{m divides index}
    Let \(A\) be a central simple algebra of index \(n\) over a field \(F\) and \(md =n\). Then \(m\) divides the index of \(\mathrm{SB}_d(A)\).
\end{lemma}
\begin{proof}
    Suppose \(\mathrm{SB}_d(D)(L) \neq \emptyset\) for some finite field extension \(L/F\). By Lemma~\ref{point,index and rational}, this implies the index of \(D_L\) divides \(d\). Let \(E/L\) be a maximal separable splitting field of \(D_L\). Since \(E\) splits \(D\), we have
\(n \mid [E:F]=[E:L][L:F]\). As \([E:L]\mid d\), it follows that \(n \mid d \cdot [L:F]\),
or equivalently that \(m=n/d\) divides \([L:F]\). Hence, \(m\) divides the index of \(\mathrm{SB}_d(D)\).
\end{proof}

\begin{theorem}\label{index of generalized SB}
Let \(X = \mathrm{SB}_{n_1,n_2,\dots,n_k}(A)\) be a Severi--Brauer flag variety over \(F\) associated with a central simple algebra \(A\) of index \(n\). The index of \(X\) is given by \(n/d\), where \(d=gcd(n_1,n_2,\dots,n_k,n)\).
\end{theorem} 

\begin{proof}
Let \(D\) denote the central division algebra over \(F\) Brauer-equivalent to \(A\). By Lemma~\ref{Severi--Brauer flag to generalized Severi--Brauer}, we know \(X\) is stably birational to \(\mathrm{SB}_d(D)\). Since index of a variety is stably birational invariant, it suffices to prove that index of \(\mathrm{SB}_d(D)\) is \(n/d\). Note that in this case, \(d\) divides the index of \(D\). Thus by \cite[Theorem 7.1]{zerocycle}, the claim holds if char(F) does not divide \(n\) or if \(F\) is perfect. Hence we assume that \(F\) is an imperfect field of characteristic \(p > 0\) and \(p\) divides n. Let \(m = n/d\). By Lemma~\ref{m divides index}, \(m\) divides the index of \(\mathrm{SB}_d(D)\). It therefore suffices to show that the index of \(\mathrm{SB}_d(D)\) divides m. \\
There exists a Henselian discrete valuation ring, say \(R\), with residue field \(F\) and fraction field \(K\) of characteristic \(0\). Further, \(\Br(R)\) is isomorphic to \(\Br(F)\), thus there exists an Azumaya algebra \(\mathcal{D}\) over \(R\) such that \(\mathcal{D} \otimes_R F \cong D\). Let \(\tilde{D}:= \mathcal{D}\otimes_R K \in \Br(K)\).\\
Note that for any \(0 < r < n\), the \(r\)-th generalized Severi--Brauer scheme \(\mathrm{SB}_r(\mathcal{D})\) over \(Spec \, R\) is proper with special fiber \(\mathrm{SB}_r(D)\) and generic fiber \(\mathrm{SB}_r(\tilde{D})\). By \cite[Lemma~2.10]{CT}, the index of  \(\mathrm{SB}_r(D)\) divides the index of \(\mathrm{SB}_r(\tilde{D})\). In particular, for \(r=1\), we obtain \(n = \mathrm{ind}(D) \mid \mathrm{ind}(\tilde{D})\). Since \(\deg(\tilde{D})=n\), it follows that \(\mathrm{ind}(\tilde{D})=n\). Thus \(\tilde{D}\) is a division algebra over \(K\). Since \(char(K)=0\), the index of \(\mathrm{SB}_d(\tilde{D})\) is \(n/d= m\) \cite[Lemma 7.1]{zerocycle}. Taking \(r=d\), we conclude that the index of \(\mathrm{SB}_d(D)\) over \(F\) divides \(m\).      
\end{proof}
\section{Torsion bounds for \(\mathrm{A_0}\) of Severi--Brauer flag varieties}\label{proof of main thm}
In this section, we prove Theorem~\ref{main theorem}, which states that \(\mathrm{A_0}(\mathrm{SB}_{n_1,n_2,\dots,n_k}(A))\) is \((d,n/d)\)-torsion where \(d = gcd(n_1,n_2 \ldots,n_k,\mathrm{ind}(A))\). We also prove Corollary~\ref{coro} and Proposition~\ref{2-primary part is trivial}. We begin by reducing the problem of determining torsion bounds to the case where \(\mathrm{ind}(A)\) is a prime power.
\begin{proposition}\label{coprime degrees}
    Let \(A, B\) be central simple algebras of coprime degrees \(m,n\) respectively. Let \(s,t\) be two positive integers such that \(s|m\) and \(t |n\). Then the varieties \(\mathrm{SB}_{st}(A \otimes_F B)\) and \(\mathrm{SB}_s(A) \times \mathrm{SB}_t(B)\) are stably birational.
\end{proposition}
\begin{proof}
Let \(L/F\) denote the function field of \(\mathrm{SB}_{st}(A \otimes_F B)\). Then \(\mathrm{SB}_{st}(A \otimes_F B)(L) \neq \emptyset\) and by Lemma \ref{point,index and rational}, we know \(\mathrm{ind}((A \otimes_F B) \otimes_F L)|st\). Since \(\mathrm{ind}(A),\mathrm{ind}(B)\) are coprime by assumption, we get \(\mathrm{ind}(A \otimes_F L) | s\) and \(\mathrm{ind}(B \otimes_F L) |t \). This in turn implies that \(\mathrm{SB}_s(A) (L) \neq \emptyset\) and \( \mathrm{SB}_t(B)(L) \neq \emptyset\). By Lemma~\ref{point,index and rational}, we further get that \((\mathrm{SB}_s(A) \times \mathrm{SB}_t(B)) \otimes_F L) \) is rational. \\
Now let \(E/F\) denote the function field of \(\mathrm{SB}_s(A) \times \mathrm{SB}_t(B)\). Since 
both \(\mathrm{SB}_s(A)\) and \(\mathrm{SB}_t(B)\) has an \(E\)-rational point, by Lemma~\ref{point,index and rational}, \(\mathrm{ind}(A \otimes_F E)\) divides \(s\) and \(\mathrm{ind}(B \otimes_F E)\) divides \(t\). This implies that \(\mathrm{ind}((A \otimes_F B) \otimes_F E)|st\) and thus \(\mathrm{SB}_{st}(A \otimes_F B) \otimes_F E\) is rational by Lemma~\ref{point,index and rational}. We conclude using Lemma~\ref{stable birationality}.
\end{proof}
We now exploit the relation between \(\mathrm{SB}_d(D)\) and the symmetric power of \(\mathrm{SB}(D)\) to show that \(\mathrm{A_0}(\mathrm{SB}_d(D))\) is \(d\)-torsion when \(d-1\) is coprime to the index of \(D\). 
\begin{lemma}\label{d-torsion}
     Let \(D\) be a central division algebra of index \(n\). Suppose \(0<d<n\) is an integer such that \(d|n\) and \((d-1,n)=1\). Then the group \(\mathrm{A_0}(\mathrm{SB}_d(D))\)  is of \(d\)-torsion.
\end{lemma}
    \begin{proof}
Let \(X\) denote the Severi--Brauer variety \(\mathrm{SB}(D)\). There exists a natural surjective morphism \(\phi:X \times \text{Sym}^{d-1}(X) \to \text{Sym}^d(X)\) of degree \(d\). By 
\cite[Theorem 1.5]{symmetricpowers}, \(\text{Sym}^{d-1}(X)\)(resp. \(\text{Sym}^{d}(X)\)) is stably birational to \(\mathrm{SB}_{d-1}(D)\) (resp. \(\mathrm{SB}_d(D)\)). Since \(d-1\) is coprime to \(n\), it follows from \cite[Proposition 4.1]{Florence} that \(\mathrm{SB}_{d-1}(D) \sim_{\text{bir}} X \times \P^{N'}\) for a suitable integer \(N'\).  Combining these results, we deduce that \(\text{Sym}^{d-1}(X)\) is birational to \(X \times \P^N \) for some integer \(N\). Thus, there exists a dominant rational map 
\(\phi^{'}:X \times X \times \P^N  \dashrightarrow \text{Sym}^d(X) \)
of degree \(d\). Since the group \(\mathrm{A_0}\) is a stable birational invariant, we obtain
\(\mathrm{A_0}(X \times X \times \P^N) \cong \mathrm{A_0}(X) = 0\), where the vanishing follows from \cite[Corollary 7.3]{MC}. Finally, applying the standard restriction-corestriction argument, we conclude that \(\mathrm{A_0}(\text{Sym}^d(X)) \cong \mathrm{A_0}(\mathrm{SB}_d(D)) \) is a \(d\)-torsion group. 
\end{proof}
\begin{lemma}\label{cellular decomposition}
Let \(X\) and \(Y\) be proper varieties over \(F\), and suppose that \(X\) admits a cellular decomposition. If \(\mathrm{A_0}(Y)\) is \(d\)-torsion for some integer \(d\), then so is \(\mathrm{A_0}(X \times Y)\).
\end{lemma}
\begin{proof}
Since \(X\) admits a cellular decomposition, the exterior product map
\[
\mathrm{CH_0}(X) \otimes \mathrm{CH_0}(Y) \xrightarrow{\ \times\ } \mathrm{CH_0}(X \times Y)
\]
is surjective \cite[Example~1.10.2]{fulton2012intersection}. Moreover, \(\mathrm{CH_0}(X) \cong \mathbb{Z}\), generated by a zero-cycle of degree one. Hence the product map induces a surjection \(\mathrm{A_0}(Y) \twoheadrightarrow \mathrm{A_0}(X \times Y)\), and therefore \(\mathrm{A_0}(X \times Y)\) is \(\mathrm{A_0}(Y)\)-torsion.
\end{proof}
We now prove Theorem~\ref{main theorem}.
\begin{proof}[Proof of Theorem~\ref{main theorem}]
Let \(D\) be the division algebra Brauer-equivalent to \(A\) over \(F\). Since \(X\) is stably birational to \(\mathrm{SB}_d(D)\) (Lemma~\ref{Severi--Brauer flag to generalized Severi--Brauer}), it suffices to show that \(\mathrm{A_0}(\mathrm{SB}_d(D))\) is \((d,n/d)\)-torsion. Since \(\mathrm{A_0}(\mathrm{SB}_d(D))\) is \(\mathrm{ind}(\mathrm{SB}_d(D))\)-torsion and \(\mathrm{ind}(\mathrm{SB}_d(D))=n/d\) by Theorem~\ref{index of generalized SB}, it remains to show that \(\mathrm{A_0}(\mathrm{SB}_d(D))\) is \(d\)-torsion. \\
Let \(D \cong \bigotimes_{i=1}^k D_i\) be the primary decomposition of \(D\) with \(\mathrm{ind}(D_i)=p_i^{a_i}\) for \(1 \leq i \leq k\) (see, Proposition~\ref{primary decomposition}). Let \(d=\prod_{i=1}^k p_i^{b_i}\). By Proposition~\ref{coprime degrees}, \(\mathrm{A_0}(\mathrm{SB}_d(D)) \cong \mathrm{A_0}\Bigl(\prod_{i=1}^k \mathrm{SB}_{p_i^{b_i}}(D_i)\Bigr)\). We proceed by induction on the number of prime factors of \(d\). If \(d=p_1^{b_1}\), then Lemma~\ref{d-torsion} shows that \(\mathrm{A_0}(\mathrm{SB}_d(D))\) is \(d\)-torsion.
Now let \(d'=\prod_{i=1}^{k-1} p_i^{b_i}\) and set \(Y=\prod_{i=1}^{k-1} \mathrm{SB}_{p_i^{b_i}}(D_i)\). By the induction hypothesis, \(\mathrm{A_0}(Y)\) is \(d'\)-torsion.\\
For each \(1 \le i \le k\), let \(L_i/F\) be a maximal separable splitting field of \(D_i\), and let \(E=L_1\cdots L_{k-1}\). Then \((Y \times \mathrm{SB}_{p_k^{b_k}}(D_k))_E \cong \Bigl(\prod_{i=1}^{k-1} \mathrm{Gr}(p_i^{b_i},p_i^{a_i})\Bigr)_E \times \mathrm{SB}_{p_k^{b_k}}(D_k \otimes_F E)\). Since \([E:F]\) is coprime to \(\mathrm{ind}(D_k)=p_k^{a_k}\), we have \(\mathrm{ind}(D_k \otimes_F E)=\mathrm{ind}(D_k)\). Hence, by Lemma~\ref{d-torsion}, the group \(\mathrm{A_0}(\mathrm{SB}_{p_k^{b_k}}(D_k \otimes_F E))\) is \(p_k^{b_k}\)-torsion. As Grassmannians admit a cellular decomposition, Lemma~\ref{cellular decomposition} implies that \(\mathrm{A_0}\bigl((Y \times \mathrm{SB}_{p_k^{b_k}}(D_k))_E\bigr)\) is \(p_k^{b_k}\)-torsion. By the restriction--corestriction argument, it follows that
\begin{equation}\label{first}
\mathrm{A_0}(Y \times \mathrm{SB}_{p_k^{b_k}}(D_k)) \text{ is } \Bigl(\prod_{i=1}^{k-1} p_i^{a_i}\Bigr)p_k^{b_k}\text{-torsion}.
\end{equation}
Since \(L_k\) splits \(D_k\), we have \((Y \times \mathrm{SB}_{p_k^{b_k}}(D_k))_{L_k} \cong Y_{L_k} \times \mathrm{Gr}(p_k^{b_k},p_k^{a_k})\). By Lemma~\ref{cellular decomposition} and the induction hypothesis, the group \(\mathrm{A_0}(Y_{L_k} \times \mathrm{Gr}(p_k^{b_k},p_k^{a_k}))\) is \(d'\)-torsion. Applying restriction--corestriction again, we obtain
\begin{equation}\label{second}
\mathrm{A_0}(Y \times \mathrm{SB}_{p_k^{b_k}}(D_k)) \text{ is } \Bigl(\prod_{i=1}^{k-1} p_i^{b_i}\Bigr)p_k^{a_k}\text{-torsion}.
\end{equation}
Taking the gcd of the bounds in \eqref{first} and \eqref{second}, we conclude that \(\mathrm{A_0}(\mathrm{SB}_d(D))\) is \(d\)-torsion.
\end{proof}

\begin{proof}[Proof of Corollary~\ref{coro}]
Since \(\mathrm{SB}_r(A)\) is stably birational to \(\mathrm{SB}_d(D)\), where \(d=(r,n)\) and \(D\) is the central division algebra Brauer-equivalent to \(A\), it suffices to show that \(\mathrm{A_0}(\mathrm{SB}_d(D))=0\). For any such \(d\), the integer \(n/d\) is coprime to \(d\) by hypothesis. It follows from Theorem~\ref{main theorem} that \(\mathrm{A_0}(\mathrm{SB}_d(D))\) is trivial.
\end{proof}
\begin{proposition}\label{2-primary part is trivial}
Let \(A\) be a central simple algebra over a field \(F\) of index \(n\). Assume that \(char(F)\nmid n\) and that \(n\) is even but not divisible by \(8\). Then \(\mathrm{A_0}(\mathrm{SB}_r(A))=0\) for every \(r\) satisfying \((r,n)=2\).
\end{proposition}

\begin{proof}
Let \(D\) be the central division algebra Brauer-equivalent to \(A\), and write \(D \cong D' \otimes \tilde{D}\), where \(D'\) is the \(2\)-primary component and \(\tilde{D}\) is the prime-to-\(2\) component of \(D\). By Lemma~\ref{Severi--Brauer flag to generalized Severi--Brauer} and Proposition~\ref{coprime degrees}, we have
\(\mathrm{A_0}(\mathrm{SB}_r(A)) \cong \mathrm{A_0}(\mathrm{SB}_2(D') \times \mathrm{SB}(\tilde{D}))\).\\
Let \(L/F\) be a maximal separable splitting field of \(\tilde{D}\). Then \(\mathrm{SB}(\tilde{D})_L \cong \mathbb{P}^m\) for some \(m\), and hence Lemma~\ref{cellular decomposition} implies that \(\mathrm{A_0}((\mathrm{SB}_2(D') \times \mathrm{SB}(\tilde{D}))_L)\) is \(\mathrm{A_0}(\mathrm{SB}_2(D')_L)\)-torsion. Since \(\mathrm{ind}(D') \mid 4\), the hypotheses on \(F\) and \cite[Theorem~7.3]{zerocycle} imply that \(\mathrm{A_0}(\mathrm{SB}_2(D')_L)=0\). Thus \(\mathrm{A_0}((\mathrm{SB}_2(D') \times \mathrm{SB}(\tilde{D}))_L)=0\), and restriction--corestriction shows that \(\mathrm{A_0}(\mathrm{SB}_2(D') \times \mathrm{SB}(\tilde{D}))\) is \([L:F]\)-torsion.\\
Now let \(E/F\) be a maximal separable splitting field of \(D'\). Since \(\mathrm{SB}_2(D')_E\) is a Grassmannian, Lemma~\ref{cellular decomposition} implies that \(\mathrm{A_0}((\mathrm{SB}_2(D') \times \mathrm{SB}(\tilde{D}))_E)\) is \(\mathrm{A_0}(\mathrm{SB}(\tilde{D})_E)\)-torsion. By \cite[Corollary~7.3]{MC}, \(\mathrm{A_0}(\mathrm{SB}(\tilde{D})_E)=0\), and hence \(\mathrm{A_0}(\mathrm{SB}_2(D') \times \mathrm{SB}(\tilde{D}))\) is \([E:F]\)-torsion. Since \([L:F]\) and \([E:F]\) are coprime, we conclude that \(\mathrm{A_0}(\mathrm{SB}_2(D') \times \mathrm{SB}(\tilde{D}))=0\).
\end{proof}
\section{Refined torsion bounds via products}\label{Refined torsion bounds via products}

In \cite[Proposition~4.1]{MC}, Merkurjev and Chernousov show that if \(X\) is a scheme over \(F\) and \(Y\) is a projective homogeneous variety satisfying \(Y(F(x)) \neq \emptyset\) for every \(x \in X\), then the projection \(\pi : X \times Y \to X\) induces an isomorphism \(\mathrm{CH_0}(X \times Y) \cong \mathrm{CH_0}(X)\). In particular, if \(\mathrm{A_0}(X \times Y)=0\), then \(\mathrm{A_0}(X)=0\). We extend this principle in Proposition~\ref{better torsion bounds} by replacing the existence of rational points with the weaker assumption that \(\mathrm{A_0}(Y_{F(x)})=0\) for all \(x \in X\), and show that \(\mathrm{A_0}(X)\) is \(\mathrm{lcm}\{\,\mathrm{ind}(Y_{F(t)}) : t \in X_{(0)}\,\}\)-torsion.

For a division algebra \(D\) of index \(p^a\), Lemma~\ref{d-torsion} gives that \(\mathrm{A_0}(\mathrm{SB}_{p^b}(D))\) is \(p^b\)-torsion for \(2 \le b \le a-1\). Assuming \(\mathrm{A_0}(\mathrm{SB}_p(D))=0\), this bound improves to \(p^{b-1}\) (Corollary~\ref{reduction in torsion}), yielding an inductive approach to the vanishing of \(\mathrm{A_0}\) for generalized Severi--Brauer varieties. As a direct application, we show that if \(A\) has exponent not divisible by \(4\), then \(\mathrm{A_0}(\mathrm{SB}_r(A))\) is \(2\)-torsion whenever \((r,\mathrm{ind}(A))=4\) (Corollary~\ref{2-torsion}).

\begin{proposition}\label{better torsion bounds}
Let \(X,Y\) be smooth projective varieties over \(F\). Assume \(\mathrm{A_0}(X \times Y) = 0\). If \(\mathrm{A_0}(Y_L)=0\) for every finite extension \(L/F\), then \(\mathrm{A_0}(X)\) is 
\(\mathrm{lcm}\{\,\mathrm{ind}(Y_{F(t)}) : t \in X_{(0)}\,\}\)-torsion.
\end{proposition}
\begin{proof}
Consider the projection \(\pi: X \times Y \to X\). The spectral sequence associated with \(\pi\) (\cite[Section 8]{Rost}) \[
E^1_{p,q}= \bigsqcup\limits_{t \in X_{(p)}}A_q(Y_{F(t)},K_p) \Rightarrow A_{p+q}(X \times Y, K_0)\]
yields an exact sequence \({E^{1}_{1,0}} \xrightarrow{\partial} {E^{1}_{0,0}} \rightarrow  \mathrm{CH_0}(X \times Y) \rightarrow 0\). Expanding the terms, we obtain a commutative diagram
\[
\begin{tikzcd}
\bigoplus\limits_{t \in X_{(1)}} K_1(Y_{F(t)}) 
\arrow[d, "f"] 
\arrow[r, "\partial"] 
& \bigoplus\limits_{t \in X_{(0)}} \mathrm{CH_0}(Y_{F(t)}) 
\arrow[d, "g"]   
\arrow[r, two heads] 
& \mathrm{CH_0}(X \times Y)
\arrow[d, "\pi_{*}"] \\
\bigoplus\limits_{t \in X_{(1)}} K_1(F(t))              
\arrow[r, "\partial"] 
& \bigoplus\limits_{t \in X_{(0)}} \mathrm{CH_0}(F(t))                                   
\arrow[r, two heads] 
& \mathrm{CH_0}(X).          
\end{tikzcd}
\]
Since the degree map \(\mathrm{CH_0}(X \times Y) \to \mathbb{Z}\) factors through \(\pi_*\), the assumption \(\mathrm{A_0}(X \times Y)=0\) implies that \(\pi_*\) is injective. Moreover, for each \(t \in X_{(0)}\), the extension \(F(t)/F\) is finite, so by assumption \(\mathrm{A_0}(Y_{F(t)})=0\). It follows that the degree map \(\mathrm{CH_0}(Y_{F(t)}) \to \mathbb{Z}\) is injective; hence, \(g\) is injective. Its cokernel is \(\bigoplus_{t \in X_{(0)}} \Z/\mathrm{ind}(Y_{F(t)})\Z\).
A diagram chase then yields a short exact sequence
\[
0 \longrightarrow \operatorname{coker}(f) \longrightarrow \operatorname{coker}(g) \longrightarrow \operatorname{coker}(\pi_*) \longrightarrow 0.
\]
Since \(\mathrm{A_0}(X) \cong \operatorname{coker}(\pi_*)\), it follows that \(\mathrm{A_0}(X)\) is annihilated by 
\(\mathrm{lcm}\{\mathrm{ind}(Y_{F(t)}) : t \in X_{(0)}\}\).
\end{proof}
\begin{remark}
The proof shows that, in Proposition~\ref{better torsion bounds}, it suffices to assume \(\mathrm{A_0}(Y_{F(t)})=0\) for every \(t \in X_{(0)}\).
\end{remark}

\begin{lemma}\label{Symmetric powers and stable birationality}
    Let \(A \) be a central simple algebra of index \(n\) over \(F\). Then for any two integers,\(1 \leq d,e \leq n-1\), \(\mathrm{A_0}(\mathrm{SB}_e(A) \times \mathrm{SB}_d(A)) \cong \mathrm{A_0}(\mathrm{SB}_{(d,e)}(A))\).
\end{lemma}
\begin{proof}
    By \cite[Theorem~1.1(6)]{Kollar}, there is a stable birational equivalence
\[
\mathrm{Sym}^e(\mathrm{SB}(A)) \times \mathrm{Sym}^d(\mathrm{SB}(A)) \sim_{\mathrm{stab}} \mathrm{Sym}^{(d,e)}(\mathrm{SB}(A)).
\]
For any \(k \le n\), \(\mathrm{Sym}^k(\mathrm{SB}(A))\) is stably birational to \(\mathrm{SB}_k(A)\) (\cite[Theorem~1.5]{symmetricpowers}). Hence
\[
\mathrm{SB}_e(A) \times \mathrm{SB}_d(D) \sim_{\mathrm{stab}} \mathrm{SB}_{(d,e)}(A).
\]
Since \(\mathrm{A_0}\) is invariant under stable birational equivalence, it follows that 
\[
\mathrm{A_0}(\mathrm{SB}_e(A) \times \mathrm{SB}_d(A)) \cong \mathrm{A_0}(\mathrm{SB}_{(d,e)}(A)).
\]
\end{proof}

\begin{remark}
We illustrate Proposition~\ref{better torsion bounds} in the case of Severi--Brauer varieties. Let \(D\) be a central division algebra over \(F\) of index \(p^a\). By \cite[Corollary~7.3]{MC}, \(\mathrm{A_0}(\mathrm{SB}(D)_L)=0\) for every finite extension \(L/F\). By Lemma~\ref{Symmetric powers and stable birationality}, \(\mathrm{A_0}(\mathrm{SB}_{p^c}(D) \times \mathrm{SB}(D))\cong \mathrm{A_0}(\mathrm{SB}(D))=0\) for every \(0 \leq c \leq a-1\). Applying Proposition~\ref{better torsion bounds} with \(X=\mathrm{SB}_{p^b}(D)\) and \(Y=\mathrm{SB}(D)\), we obtain the commutative diagram
\begin{center}
\begin{tikzcd}
\bigoplus\limits_{t \in X_{(1)}} K_1(D_{F(t)})
\arrow[d, "f"] 
\arrow[r, "\partial"] 
& \bigoplus\limits_{t \in X_{(0)}} \mathrm{ind}(D_{F(t)})\mathbb{Z} 
\arrow[d, "g", hook]   
\arrow[r, two heads] 
& \mathrm{CH_0}(X \times Y)
\arrow[d, "\pi_{*}", hook] \\
\bigoplus\limits_{t \in X_{(1)}} (F(t))^{\times}              
\arrow[r, "\partial"] 
& \bigoplus\limits_{t \in X_{(0)}} \mathbb{Z}                                   
\arrow[r, two heads] 
& \mathrm{CH_0}(X).          
\end{tikzcd}
\end{center}
and a short exact sequence
\[
0 \to \operatorname{coker}(f) \to \operatorname{coker}(g) \to \mathrm{A_0}(\mathrm{SB}_{p^b}(D)) \to 0.
\]
\end{remark}

\begin{corollary}\label{reduction in torsion}
Let \(D\) be a central division algebra over \(F\) of index \(p^a\). If \(\mathrm{A_0}(\mathrm{SB}_p(D_L))=0\) for every finite extension \(L/F\), then for each integer \(b\) with \(2 \leq b \leq a-1\), the group \(\mathrm{A_0}(\mathrm{SB}_{p^b}(D))\) is \(p^{b-1}\)-torsion.
\end{corollary}

\begin{proof}
Let \(\mathrm{SB}(D)\) denote the Severi--Brauer variety associated with \(D\). By Lemma~\ref{Symmetric powers and stable birationality}, we obtain \(\mathrm{A_0}(\mathrm{SB}_{p^c}(D) \times \mathrm{SB}_p(D)) \cong \mathrm{A_0}(\mathrm{SB}_p(D))\) for every \(2 \leq c \leq a-1\). Since \(\mathrm{A_0}(\mathrm{SB}_p(D))=0\) by assumption, it follows that \(\mathrm{A_0}(\mathrm{SB}_{p^b}(D) \times \mathrm{SB}_p(D))=0\) for all \(2 \le b \le a-1\). Set \(X=\mathrm{SB}_{p^b}(D)\) and \(Y=\mathrm{SB}_p(D)\). By Proposition~\ref{better torsion bounds}, \(\mathrm{A_0}(X)\) is annihilated by \(\mathrm{lcm}\{\,\mathrm{ind}(Y_{F(t)}) : t \in X_{(0)}\,\}\). By Lemma~\ref{point,index and rational}, for \(t \in X_{(0)}\) we have \(\mathrm{ind}(D_{F(t)}) \mid p^b\), so \(\mathrm{ind}(D_{F(t)}) \in \{1,p,\dots,p^b\}\). By Theorem~\ref{index of generalized SB}, \(\mathrm{ind}(Y_{F(t)}) = \mathrm{ind}(\mathrm{SB}_p(D_{F(t)})) = \mathrm{ind}(D_{F(t)})/p\), so \(\mathrm{ind}(Y_{F(t)}) \in \{1,p,\dots,p^{b-1}\}\). Hence the lcm is \(p^{b-1}\), and it follows that \(\mathrm{A_0}(\mathrm{SB}_{p^b}(D))\) is \(p^{b-1}\)-torsion.
\end{proof}

As a consequence of Corollary~\ref{reduction in torsion}, we obtain the following.
\begin{corollary}\label{2-torsion}
   Let \(A\) be a central simple algebra over \(F\) of index \(n\) whose exponent is not divisible by \(4\). Then \(\mathrm{A_0}(\mathrm{SB}_r(A))\) is \(2\)-torsion whenever \((r,n)=4\). 
\end{corollary}

\begin{proof}
By Lemma~\ref{Severi--Brauer flag to generalized Severi--Brauer}, \(\mathrm{SB}_r(A)\) is stably birational to \(\mathrm{SB}_4(D)\), where \(D\) is the division algebra Brauer-equivalent to \(A\). Let \(D \simeq D' \otimes \tilde{D}\), where \(D'\) is the \(2\)-primary component and \(\tilde{D}\) has index prime to \(2\). By Proposition~\ref{coprime degrees}, \(\mathrm{SB}_4(D) \sim_{\mathrm{stab}} \mathrm{SB}_4(D') \times \mathrm{SB}(\tilde{D})\), and hence \(\mathrm{A_0}(\mathrm{SB}_r(A)) \cong \mathrm{A_0}\bigl(\mathrm{SB}_4(D') \times \mathrm{SB}(\tilde{D})\bigr)\).
Since the exponent of \(D'\) is not divisible by \(4\), the same holds for \(D'_L\) over any finite extension \(L/F\). By \cite[Corollary~8.4]{MC}, we have \(\mathrm{A_0}(\mathrm{SB}_2(D'_L))=0\) for all such \(L\). Applying Corollary~\ref{reduction in torsion} with \(p=2\) and \(b=2\), it follows that \(\mathrm{A_0}(\mathrm{SB}_4(D'))\) is \(2\)-torsion.
As in the proof of Theorem~\ref{main theorem}, passing to splitting fields of \(D'\) and \(\tilde{D}\) and using restriction--corestriction, we see that 
\(\mathrm{A_0}\bigl(\mathrm{SB}_4(D') \times \mathrm{SB}(\tilde{D})\bigr)\)
is both \(\mathrm{ind}(D')\)-torsion and \(2\,\mathrm{ind}(\tilde{D})\)-torsion. Hence it is \(\gcd\bigl(\mathrm{ind}(D'),\, 2\,\mathrm{ind}(\tilde{D})\bigr)=2\)-torsion. This proves the claim.
\end{proof}

\section{Index reduction over local and global fields}\label{Index reduction over local and global fields}
Let \(F\) be a local field. It is well known that there is a canonical isomorphism \(\operatorname{inv}_F \colon \Br(F) \xrightarrow{\sim} \mathbb{Q}/\mathbb{Z}\). For any finite extension \(L/F\), the restriction map \(\operatorname{Res}_{L/F} \colon \Br(F) \to \Br(L)\) satisfies \(\operatorname{inv}_L \circ \operatorname{Res}_{L/F} = [L:F] \cdot\operatorname{inv}_F\) \cite[Chapter 13, Proposition 7]{serre2013local}. In particular, for \(A \in \Br(F)\), we have \(\mathrm{ind}(A_L)=\mathrm{ind}(A)/[L:F]\). This allows us to construct degree \(p\) extensions with controlled index behavior: given two degree \(p\) extensions \(L_0/F\) and \(L_1/F\) with \(\mathrm{ind}(A_{L_i})=p^{m-1}\), one can find a third degree \(p\) extension \(K/F\), such that \(\mathrm{ind}(A_K)=p^{m-1}\) and \(\mathrm{ind}(A_{KL_i})=p^{m-2}\) for \(i=0,1\). A key input is that a non-archimedean local field admits at least \(p+1\) distinct extensions of degree \(p\) (Lemma~\ref{atleast 3 distinct deg p}). \\
For global fields, this construction is more subtle. Using a global realisation theorem (Theorem~\ref{Conrad}), we patch local data to produce a global extension with the required properties; this is carried out in Lemma~\ref{local-global}, which plays a crucial role in the proof of Lemma~\ref{p-special connectedness-global}.
\begin{lemma} \label{atleast 3 distinct deg p}
Let \(F\) be a non-archimedean local field and \(p\) a prime. Then there exist at least \(p+1\) distinct extensions of \(F\) of degree \(p\) inside a fixed separable closure \(F^s\). 
\end{lemma}
\begin{proof}
The proof naturally divides into three cases depending on the characteristic of \(F\) and the presence of \(p\)-th roots of unity.\\
Case 1: \(\operatorname{char}(F)= p\). \\
In this case, \(F \cong \mathbb{F}_q((t))\), where \(q\) is a power of \(p\). By Artin--Schreier theory, every cyclic extension of degree \(p\) arises from an equation of the form \(x^p-x-a\), \(a \in F\). Such extensions correspond to the nontrivial \(1\)-dimensional subspaces of the \(\mathbb{F}_p\)-vector space \(F / \wp(F)\), where \(\wp(x)=x^p-x\). Elements of the form \(c t^{-k}\), with \(c \in \mathbb{F}_q^\times\) and \(k>0\) coprime to \(p\), yield linearly independent classes in \(F/\wp(F)\). Hence this vector space is infinite-dimensional, and in particular there are infinitely many degree \(p\) extensions.\\
Case 2: \(\operatorname{char}(F)\neq p\) and \(\zeta_p \in F\).\\
By Kummer theory, degree \(p\) extensions correspond to 1-dimensional subspaces of the \(\mathbb{F}_p\)-vector space \( F^\times / (F^\times)^p\). Using the structure of the multiplicative group of a local field (see \cite[Proposition 5.7]{neukirch2010algebraic}), we get \(\dim_{\mathbb{F}_p}\bigl(F^\times / (F^\times)^p\bigr) \ge 2\). Hence, the number of degree \(p\) extensions is at least \(p^2-1/p-1\ = p+1\).\\
Case 3: \(\operatorname{char}(F)\neq p\) and \(\zeta_p \notin F\). \\
There exists a unique unramified extension of degree \(p\). Let \(\pi\) be a uniformizer of \(F\). The Eisenstein polynomial \(f(x)=x^p-\pi\) is irreducible over \(F\). Let \(\alpha_1,\dots,\alpha_p\) be its roots in \(F^s\). Each field \(F(\alpha_i)\) is a totally ramified extension of degree \(p\). If \(F(\alpha_i)=F(\alpha_j)\) for some \(i \ne j\), then \(\alpha_i / \alpha_j = \zeta_p^k \in F(\alpha_i) \) for some \(1 \leq k \leq p-1\). Thus
\(F(\zeta_p^k) \subseteq F(\alpha_i)\). Since \([F(\alpha_i):F]=p\), we must have \([F(\zeta_p^k):F] \mid p\). On the other hand, \([F(\zeta_p^k):F]\mid (p-1)\), so it must be \(1\), contradicting \(\zeta_p \notin F\). Thus the fields \(F(\alpha_i)\) are pairwise distinct, giving \(p+1\) distinct degree \(p\) extensions, including the unramified extension.
\end{proof}

We now recall some basic facts about Brauer groups of global fields. 
\begin{theorem}\cite[Theorem 1.5.36]{poonen2023rational}\label{ABHN}
Let \(F\) be a global field. For each place \(v\) of \(F\), let \(F_v\) denote the completion of \(F\), and let \(inv_v:\Br\,(F_v) \to \mathbb{Q}/\mathbb{Z}\) be the invariant map.

\begin{enumerate}
\item Then the sequence
\[
0 \longrightarrow \Br \,(F) \longrightarrow \bigoplus_{v} \Br \,(F_v)
\xrightarrow{\ \sum \operatorname{inv}_v\ } \mathbb{Q}/\mathbb{Z} \longrightarrow 0
\]
is exact.
\item Every element of \(\Br\, (F)\) has period equal to index.
\end{enumerate}
\end{theorem}
Since \(\Br\,(F)\) injects into \(\bigoplus_{v} \Br \, (F_v)\), it follows that for any \(A \in \Br \, (F)\), the class \(A \otimes_F F_v \in \Br \, (F_v)\) is trivial for all but finitely many places \(v\) of \(F\). In particular, since \(\mathrm{ind}(A)=\mathrm{per}(A)\), the set of places \(v\) of \(F\) for which \(\mathrm{ind}(A \otimes_F F_v) \neq 1\) is finite.
As a consequence, we obtain the following corollary.
\begin{corollary}\cite[Corollary 18.6]{pierce2012associative}\label{global index is lcm of local}
Let \(F\) be a global field and let \(A \in \Br(F)\). Then the index of \(A\) is given by \(\mathrm{ind}(A) = lcm\{\mathrm{ind}(A \otimes_F F_v)\}\), where \(v\) runs over all places of \(F\).
\end{corollary}

The following theorem ensures that prescribed finite extensions of completions at finitely many places can be realized globally. After fixing separable closures of \(F\) and of each \(F_{v_i}\), together with embeddings \(\phi_i : F^s \hookrightarrow F_{v_i}^s\), we identify the resulting isomorphic completions, and hence write equalities in place of the isomorphisms appearing in the original theorem.

\begin{theorem}[Conrad, \cite{ConradLocalGlobal}]\label{Conrad}
Let \(F\) be a global field and let \(v_1,\dots,v_n\) be inequivalent places of \(F\). For each \(i\), let \(F_{v_i}\) be the completion of \(F\) at \(v_i\), fix a separable closure \(F_{v_i}^s\) of \(F_{v_i}\), and choose an embedding \(\phi_i : F^s \hookrightarrow F_{v_i}^s\). Let \(K_i/F_{v_i}\) be finite separable extensions. Then there exists a finite separable extension \(K/F\) and places \(w_i\) of \(K\) lying over \(v_i\) such that the natural embedding of \(K_{w_i}\) into \(F_{v_i}^s\) identifies \(K_{w_i}\) with \(K_i\) for each \(i\). Moreover, if \([K_i : F_{v_i}] = d\) for all \(i\), then \(K/F\) can be chosen so that \([K : F] = d\).
\end{theorem}

\begin{lemma}\label{local-global}
Let \(F\) be a local or global field, fix a separable closure \(F^s\), and let \(A \in \Br(F)\) be of index \(p^m\). Suppose \(L_0/F\) and \(L_1/F\) are distinct extensions of degree \(p\) such that \(\mathrm{ind}(A_{L_0})=\mathrm{ind}(A_{L_1})=p^{m-1}\). Then there exists a degree \(p\) extension \(K/F\), distinct from both \(L_0\) and \(L_1\), such that
\[
\mathrm{ind}(A_K)=p^{m-1} \quad \text{and} \quad \mathrm{ind}(A_{K L_i})=p^{m-2} \text{ for } i=0,1.
\]
Moreover, for each \(0 \le k \le m\), there exists an extension \(K'/F\) of degree \(p^k\) such that \(\mathrm{ind}(A_{K'}) = p^{m-k}\).
\end{lemma}

\begin{proof}
Suppose \(F\) is a local field. Let \(K/F\) be a degree \(p\) extension distinct from both \(L_0/F\) and \(L_1/F\) inside \(F^s\), which exists by Lemma~\ref{atleast 3 distinct deg p}. Then \([KL_0:F]=p^2=[KL_1:F]\). For any finite extension \(L/F\), the relation \(\operatorname{inv}_L \circ \operatorname{Res}_{L/F} = [L:F]\cdot \operatorname{inv}_F\) implies that \(\mathrm{ind}(A_L) = \mathrm{ind}(A)/[L:F]\). Thus we have \(\mathrm{ind}(A_K)= \mathrm{ind}(A)/[K:F]=p^{m-1}\) and \(\mathrm{ind}(A_{K L_i})= \mathrm{ind}(A)/[KL_i:F]= p^{m-2}\) for \(i=0,1\). For the second statement, for a fixed \(0 \le k \le m\), let \(K'/F\) be any extension of degree \(p^k\). By compatibility of restriction and invariant maps, we have \(\mathrm{ind}(A_{K'}) = \mathrm{ind}(A)/[K':F]=p^{m-k}\).\\
Now suppose \(F\) is a global field. Let \(L_0=F(\alpha)\) and \(L_1=F(\beta)\) for some \(\alpha,\beta\) separable over \(F\). By Corollary~\ref{global index is lcm of local}, we have \(\mathrm{ind}(A_{F_v}) \mid p^m\) for all places \(v\) of \(F\), and \(\mathrm{ind}(A_{F_{v_0}})=p^m\) for some place \(v_0\). Let \(S=\{\, v \in \Omega_F \mid \mathrm{ind}(A_{F_v}) \neq 1 \,\}\). By Theorem~\ref{ABHN}, the set \(S\) is finite. For each \(v \in S\), let \(F_{v}\) be the completion of \(F\) at \(v\), fix a separable closure \(F_{v}^s\) of \(F_{v}\), and choose an embedding \(\phi_v : F^s \hookrightarrow F_{v}^s\). For each \(v \in S\) such that at least one of \(L_0 \otimes_F F_v\) or \(L_1 \otimes_F F_v\) is a field extension of \(F_v\) of degree \(p\), choose a degree \(p\) extension \(K_v/F_v\) inside \(F_v^s\), distinct from both \(L_0 \otimes_F F_v\) and \(L_1 \otimes_F F_v\); this is possible by Lemma~\ref{atleast 3 distinct deg p}. In particular, \(K_{v_0}/F_{v_0}\) is distinct from both \(L_0 \otimes_F F_{v_0}\) and \(L_1 \otimes_F F_{v_0}\) inside \(F_{v_0}^s\). For the remaining places \(v \in S\), choose any degree \(p\) extension \(K_v/F_v\).\\
By Theorem~\ref{Conrad}, there exists a global extension \(K=F(\gamma)/F\) of degree \(p\) and, for each \(v \in S\), a place \(w\) of \(K\) lying over \(v\) such that the natural embedding of \(K_w\) into \(F_v^s\) identifies \(K_w\) with \(K_v\). By local invariant theory, for each such place \(w\) above \(v \in S\), we have \(\mathrm{ind}(A_{K_w}) = \mathrm{ind}(A_{F_v})/p\), so \(\mathrm{ind}(A_{K_w}) \mid p^{m-1}\), with equality at the place above \(v_0\). Hence \(\mathrm{ind}(A_K)=p^{m-1}\). By construction, \(K/F\) is distinct from both \(L_0/F\) and \(L_1/F\), since their completions differ inside \(F_{v_0}^s\). Thus \([KL_i:F]=p^2\) for \(i=0,1\), and the same local argument gives \(\mathrm{ind}(A_{KL_i})=p^{m-2}\).\\
For the second statement, fix \(0 \le k \le m\). For each \(v \in S\), choose an extension \(K'_v/F_v\) of degree \(p^k\) inside \(F_v^s\). By Theorem~\ref{Conrad}, there exists a global extension \(K'/F\) of degree \(p^k\) such that \(K' \otimes_F F_v = K'_v\) for all \(v \in S\). Then \(\mathrm{ind}(A_{K'_w}) = \mathrm{ind}(A_{F_v})/p^k\) for each place \(w\) of \(K'\) above \(v \in S\), and hence \(\mathrm{ind}(A_{K'})=p^{m-k}\).
\end{proof}

\section{Triviality over local and global fields}\label{proof of Vanishing theorem}
In this final section, we prove that \(\mathrm{A_0}(X)=0\) for Severi--Brauer flag varieties \(X\) over local and global fields (Theorem~\ref{vanishing over local and global}). We focus on the group \(\mathrm{A_0}(\mathrm{SB}_{p^k}(D))\) for a central division algebra \(D\) of degree \(p^m\) over \(F\) and use the theory of equivalence of fields introduced in \cite{MC}. In \cite{MC}, the main results on the triviality of \(\mathrm{A_0}(X)\) are proved using the connectedness of classes of fields. We extract from this a criterion (Lemma~\ref{main lemma for triviality}) based solely on the notion of equivalence of fields, which we use to establish the vanishing of \(\mathrm{A_0}(X)\).
\subsection{Special fields \cite[Section 5]{MC}} 
Let \(p\) be a prime number. A field \(F\) is called \(p\)-special if the degree of every finite field extension of \(F\) is a power of \(p\). We say that \(F\) is special if it is \(p\)-special for some prime \(p\). We note that if \(X\) is a scheme over a \(p\)-special field \(F\), then the index of \(X\) is a power of \(p\), and there exists a closed point \(x \in X\) such that \(\deg(x) = \mathrm{ind}(X)\). 
\begin{lemma}\label{p-degree in special field}
Let \(F\) be a \(p\)-special field for some prime \(p\) and \(F^s\) denote its separable closure. Then, for any algebraic extension \(L/F\), there exists a subextension of degree \(p\). Further, for any \(k \geq 1\) with \(p^k \leq [F^s:F]\), there exists an extension \(M/F\) of degree \(p^k\).
\end{lemma}

\begin{proof}
Let \(\tilde{L}\) be the Galois closure of \(L/F\), and set
\(G:=\Gal(\tilde{L}/F)\). The extension \(L/F\) corresponds to the subgroup \(H=\Gal(\tilde{L}/L)\subseteq G\). Let \(N\) be a proper maximal subgroup of \(G\) containing \(H\). Since \(F\) is \(p\)-special, \(G\) is a pro-\(p\) group. As every maximal subgroup of a pro-\(p\) group has index \(p\) \cite[Lemma 22.7.4]{fried2008field}, \(N\) has index \(p\) in \(G\). Hence the fixed field \(\tilde{L}^{N}\) is a degree \(p\) subextension of \(\tilde{L}/F\). Taking \(L=F^s\), the above argument shows that \(F\) admits a degree \(p\) extension \(F_1/F\). Iterating this construction yields a tower \(F \subset F_1 \subset F_2 \subset \cdots \subset F_k\) with \([F_i:F_{i-1}]=p\) for each \(i\), and hence \([F_k:F]=p^k\). This proves the second statement.
\end{proof}

\subsection{Equivalence of fields \cite[Section 6]{MC}}
Let \(F\) be a field and \(X\) be a scheme of finite type over \(F\). Let \(\mathcal{A}(X)\) denote the class of all field extensions of \(E/F\) such that \(X(E)\neq\emptyset\). For each \(a \in F\), let \(v_a\) denote the discrete valuation of the rational function field \(F(t)\) corresponding to the irreducible polynomial \(f(t)=t-a\). Two fields \(L\) and \(L'\) in \(\mathcal{A}(X)\) of the same degree \(n\) over \(F\) are called simply \(X\)-equivalent if there exists a degree \(n\) field extension \(E/F(t)\) with \(E \in \mathcal{A}(X)\), together with two discrete valuations of \(E\) lying above \(v_0\) and \(v_1\), whose residue fields are isomorphic to \(L\) and \(L'\) respectively over \(F\). If \(L\) and \(L'\) have a common subfield \(K/F\) and are simply \(X_K\)-equivalent, then they are simply \(X\)-equivalent.
Two fields \(L\) and \(L'\) in \(\mathcal{A}(X)\) are called \(X\)-equivalent if there exists a chain of fields \(L=L_0,L_1,\cdots,L_r = L'\) in \(\mathcal{A}(X)\) such that \(L_i\) is simply \(\mathcal{A}(X)\)-equivalent to \(L_{i+1}\) for all \(i = 0,1,\ldots,r-1\). The following lemma is extracted from \cite{MC} and serves as the anchor to prove Theorem~\ref{vanishing over local and global}.

\begin{lemma}\label{main lemma for triviality} 
Let \(X\) be a proper scheme of finite type over \(F\) such that \(\mathrm{A_0}(X_L)=0\) for every field \(L \in \mathcal{A}(X)\). Assume that for every special field \(F'\) algebraic over \(F\), any two fields \(L_0,L_1 \in \mathcal{A}(X_{F'})\) of degree \(\mathrm{ind}(X_{F'})\) are \(X_{F'}\)-equivalent. Then \(\mathrm{A_0}(X)=0\).
\end{lemma}
\begin{proof} 
Fix a prime \(p\), and let \(F'/F\) be the fixed field of a Sylow \(p\)-subgroup of \(\mathrm{Gal}(F^s/F)\). By assumption, the residue fields of any two points of \(X_{F'}\) of degree \(\mathrm{ind}(X_{F'})\) are \(X_{F'}\)-equivalent. Hence, by \cite[Lemma~6.4]{MC}, any two such points are rationally equivalent in \(\mathrm{CH_0}(X_{F'})\). It follows from \cite[Lemma~5.2]{MC} that \(\mathrm{A_0}(X_{F'})=0\). In particular, \(\mathrm{A_0}(X)\) is \(p\)-torsion free. Since this holds for every prime~\(p\), and \(\mathrm{A_0}(X)\) is a torsion group, we conclude that \(\mathrm{A_0}(X)=0\).
\end{proof}

\subsection{Vanishing results}
We first reduce the vanishing problem to the case of prime-power index, which is valid over any field.
\begin{lemma}\label{reduction to index p CDA}
Let \(A\) be a central simple algebra of index \(n\) over a field \(F\), and let \(D\) be the central division algebra Brauer-equivalent to \(A\). Suppose that \(D \simeq D_1 \otimes D_2 \otimes \cdots \otimes D_k\) is the primary decomposition of \(D\), where \(\operatorname{ind}(D_i)=p_i^{a_i}\) for \(1 \le i \le k\). Assume that \(\mathrm{A_0}\big(\mathrm{SB}_{p_i^{b_i}}(D_i)\big)=0\) for all \(1 \le b_i \le a_i\) and \(1 \le i \le k\). Then \(\mathrm{A_0}\big(\mathrm{SB}_r(A)\big)=0\) for all \(r\).
\end{lemma}

\begin{proof}
Fix \(r < \deg(A)\). By Lemma~\ref{Severi--Brauer flag to generalized Severi--Brauer}, \(\mathrm{SB}_r(A)\) is stably birational to \(\mathrm{SB}_{(r,n)}(A)\). Let \((r,n)=\prod_{i=1}^k p_i^{b_i}\). By Proposition~\ref{coprime degrees}, \(\mathrm{SB}_{(r,n)}(A) \sim_{\mathrm{stab}} \prod_{i=1}^k \mathrm{SB}_{p_i^{b_i}}(D_i)\). Since \(\mathrm{A_0}\) is a stable birational invariant, 
\(
\mathrm{A_0}(\mathrm{SB}_r(A)) \cong \mathrm{A_0}\!\big(\prod_{i=1}^k \mathrm{SB}_{p_i^{b_i}}(D_i)\big)\). For each \(i\), let \(L_i/F\) be a maximal separable splitting field of \(D_i\), and for each \(j\), let \(E_j\) be the compositum of the \(L_i\) with \(i\neq j\). Then over \(E_j\), all \(D_i\) with \(i\neq j\) split, so 
\(
\big(\prod_{i=1}^k \mathrm{SB}_{p_i^{b_i}}(D_i)\big)_{E_j}
\cong
\big(\prod_{i\ne j} \mathrm{Gr}(\cdot)\big)\times \mathrm{SB}_{p_j^{b_j}}(D_j).
\)
By Lemma~\ref{cellular decomposition}, 
\(
\mathrm{A_0}\!\big(\big(\prod_{i=1}^k \mathrm{SB}_{p_i^{b_i}}(D_i)\big)_{E_j}\big)
\)
is \(\mathrm{A_0}(\mathrm{SB}_{p_j^{b_j}}(D_j))\)-torsion, hence zero by assumption. By restriction--corestriction, 
\(
\mathrm{A_0}\!\big(\prod_{i=1}^k \mathrm{SB}_{p_i^{b_i}}(D_i)\big)
\)
is annihilated by \([E_j:F]\) for each \(j\). Since \(\gcd\{[E_j:F]\mid 1\le j\le k\}=1\), we obtain 
\(
\mathrm{A_0}\!\big(\prod_{i=1}^k \mathrm{SB}_{p_i^{b_i}}(D_i)\big)=0,
\)
and hence \(\mathrm{A_0}(\mathrm{SB}_r(A))=0\).
\end{proof}

\begin{lemma}\label{p-special connectedness-global}
Let \(F\) be a local or global field, and let \(F'/F\) be a \(p\)-special algebraic extension. Let \(Y = \mathrm{SB}_{p^k}(D)\) be the generalized Severi--Brauer variety associated to a division algebra \(D \in \Br(F')\) of index \(p^m\). Then any two fields \(L_0, L_1 \in \mathcal{A}(Y)\) of degree \(p^{m-k}\) over \(F'\) are \(Y\)-equivalent.
\end{lemma}

\begin{proof}
Set \(n = m-k\). We proceed by induction on \([L_0 \cap L_1 : F']\). For \(1 \le \ell \le n\), let \(P(\ell)\) denote the statement that \(L_0\) and \(L_1\) are \(Y\)-equivalent whenever \([L_0 \cap L_1 : F'] = p^{n-\ell}\). First assume \([L_0 \cap L_1 : F'] = p^{n-1}\), and let \(K = L_0 \cap L_1\). Then \([L_i:K]=p\) for \(i=0,1\). Let \(L_0 = K(\alpha)\) and \(L_1 = K(\beta)\) for some separable \(\alpha,\beta\) over \(K\). Since \(L_0, L_1 \in \mathcal{A}(Y)\) have degree \(p^n\) over \(F'\), we have \(\mathrm{ind}(D_{L_0})=\mathrm{ind}(D_{L_1})=p^k\), and hence \(\mathrm{ind}(D_K)=p^{k+1}\). \\
As \(K/F\) is algebraic, there exists a finite extension \(\tilde{F}/F\) and \(\tilde{D} \in \Br(\tilde{F})\) such that \(\tilde{D} \otimes_{\tilde{F}} K \simeq D_K\). Without loss of generality, we may assume that the minimal polynomials of \(\alpha,\beta\) are defined over \(\tilde{F}\). Thus we have \(L_0 = K\tilde{F}(\alpha)\) and \(L_1 = K\tilde{F}(\beta)\). \\ Applying Lemma~\ref{local-global} to \(\tilde{D}\) and the extensions \(\tilde{F}(\alpha)/\tilde{F}\), \(\tilde{F}(\beta)/\tilde{F}\), we obtain a degree \(p\) extension \(\tilde{K}/\tilde{F}\) such that \(\mathrm{ind}(\tilde{D}_{\tilde{K}})=p^k\) and \(\mathrm{ind}(\tilde{D}_{\tilde{K}\tilde{F}(\alpha)})=\mathrm{ind}(\tilde{D}_{\tilde{K}\tilde{F}(\beta)})=p^{k-1}\). Setting \(K' = F\tilde{K}\), we get \(\mathrm{ind}(D_{K'})=p^k\) and \(\mathrm{ind}(D_{K' L_i})=p^{k-1}\) for \(i=0,1\). Applying second part of Lemma~\ref{local-global}, we see that there exist extensions \(M_0/K' L_0\) and \(M_1/K' L_1\) of degree \(p^{k-1}\) such that \(\mathrm{ind}(D_{M_0})=\mathrm{ind}(D_{M_1})=1\).\\
Let \(K'=K(\gamma)\), and consider the interpolation polynomial \(f(t,x)=(1-t)g(x)+t\,s(x)\in K[t,x]\), where \(g,s\) are the minimal polynomials of \(\alpha,\gamma\) respectively. Let \(E=K(t)[x]/(f)\). Then \(E/K(t)\) is a degree \(p\) extension with specializations \(L_0\) and \(K'\) at \(t=0,1\), respectively. Since \([M_0(t):E]=p^k\), and \(M_0(t)\) splits \(D\), we have \(\mathrm{ind}(D_E)=p^k\), so \(E \in \mathcal{A}(Y_K)\). Hence \(L_0 \sim_{Y_K} K'\), and therefore \(L_0 \sim_Y K'\). Similarly, \(L_1 \sim_Y K'\), and thus \(L_0 \sim_Y L_1\). This proves \(P(1)\).\\
Now assume \(P(s)\) holds for all \(1 \le s \le \ell\), and suppose \([L_0 \cap L_1 : F'] = p^{n-\ell-1}\). Let \(K=L_0 \cap L_1\), and let \(L_0'/K\), \(L_1'/K\) be degree \(p\) subextensions of \(L_0/K\), \(L_1/K\), respectively, which exists by Lemma~\ref{p-degree in special field}. Then \(\mathrm{ind}(D_{L_0'})=\mathrm{ind}(D_{L_1'})=p^{k+\ell}\).\\ By Lemma~\ref{local-global}, there exists a degree \(p\) extension \(K'/K\) such that \(\mathrm{ind}(D_{K'})=p^{k+\ell}\) and \(\mathrm{ind}(D_{K' L_0'})=\mathrm{ind}(D_{K' L_1'})=p^{k+\ell-1}\). Applying Lemma~\ref{local-global} again, we obtain extensions \(M_0/K' L_0'\), \(M_1/K' L_1'\) of degree \(p^{\ell-1}\) such that \(\mathrm{ind}(D_{M_0})=\mathrm{ind}(D_{M_1})=p^k\). Thus \(M_0, M_1 \in \mathcal{A}(Y)\), and both have degree \(p^n\) over \(F'\). By construction, \([L_0 \cap M_0 : F'] = [L_0':F']= p^{n-\ell}\), so \(L_0 \sim_Y M_0\) by \(P(\ell)\). Similarly, \(L_1 \sim_Y M_1\). Moreover, \(K' \subset M_0 \cap M_1\), so \([M_0 \cap M_1 : F'] \ge p^{n-\ell}\), and hence \(M_0 \sim_Y M_1\) by the induction hypothesis. By transitivity, \(L_0 \sim_Y L_1\), completing the induction.
\end{proof}

Now we prove our main vanishing result.

\begin{proof}[Proof of Theorem~\ref{vanishing over local and global}]
By Lemma~\ref{Severi--Brauer flag to generalized Severi--Brauer}, \(X\) is stably birational to \(\mathrm{SB}_d(D')\) where \(D'\) is the central division algebra over \(F\) Brauer-equivalent to \(A\) and \(d =gcd(n_1,n_2,\dots,n)\). Since \(\mathrm{A_0}\) is stably birational invariant, it is enough to prove that \(\mathrm{A_0}(\mathrm{SB}_d(D'))\) is trivial. By Lemma~\ref{reduction to index p CDA}, it suffices to show that \(\mathrm{A_0}\bigl(\mathrm{SB}_{p^k}(D)\bigr)=0\), where \(D\) is the \(p\)-primary component of \(D'\), say of index \(p^m\). Let \(Y = \mathrm{SB}_{p^k}(D)\). By Theorem~\ref{main theorem}, the group \(\mathrm{A_0}(Y)\) is \(q\)-torsion free for every prime \(q \neq p\). Let \(F'/F\) be the minimal \(p\)-special extension. Then, by Theorem~\ref{index of generalized SB}, \(\mathrm{ind}(Y)=\mathrm{ind}(D)/{p^k}=p^{m-k}\). By Lemma~\ref{p-special connectedness-global}, any two fields \(L_0, L_1 \in \mathcal{A}(Y_{F'})\) of degree \(p^{m-k}\) over \(F'\) are \(Y_{F'}\)-equivalent. By Lemma~\ref{main lemma for triviality}, \(\mathrm{A_0}(Y_{F'})\) is trivial. It follows that \(\mathrm{A_0}(Y)\) is \(p\)-torsion free, hence trivial.
\end{proof}

\begin{remark}
In \cite{CTCONJ}, Colliot--Th\'{e}l\`ene conjectured that, for any geometrically rationally connected variety \(X\) over a \(p\)-adic field, the group \(\mathrm{A_0}(X)\) is finite. Theorem~\ref{vanishing over local and global} provides new classes of varieties for which the conjecture holds.
\end{remark}

\bibliographystyle{plain}
\bibliography{ref}

\end{document}